\documentclass{amsart}

\usepackage{amssymb, fullpage, mathrsfs, enumerate, mathtools, tikz-cd, mathpazo, hyperref}
\usepackage{cleveref}

\newcommand{\CC}{\mathbf{C}}
\newcommand{\HH}{\mathbf{H}}
\newcommand{\QQ}{\mathbf{Q}}
\newcommand{\RR}{\mathbf{R}}
\newcommand{\ZZ}{\mathbf{Z}}
\newcommand{\FF}{\mathbf{F}}
\newcommand{\DD}{\mathbf{D}}

\newcommand{\cO}{\mathcal{O}}

\newcommand{\cL}{\mathcal{L}}

\newcommand{\pp}{\mathfrak{p}}

\newcommand{\cX}{\mathcal{X}}

\newcommand{\sF}{\mathscr{F}}
\newcommand{\dR}{\mathrm{dR}}
\newcommand{\DdR}{\DD_{\dR}}
\newcommand{\Zp}{\ZZ_p}
\newcommand{\Qp}{\QQ_p}

\newcommand{\into}{\hookrightarrow}
\newcommand{\et}{\text{\textup{\'et}}}

\newcommand{\plec}{\mathrm{plec}}
\newcommand{\mul}{\mathrm{mul}}

\DeclareMathOperator{\GL}{GL}
\DeclareMathOperator{\ord}{ord}
\DeclareMathOperator{\Gal}{Gal}
\DeclareMathOperator{\Gr}{Gr}
\DeclareMathOperator{\Fil}{Fil}
\newcommand{\BGG}{\mathcal{B}^\bullet_\mu}

\theoremstyle{plain}
    \newtheorem{theorem}{Theorem}[section]
    \newtheorem{lemma}[theorem]{Lemma}
    \newtheorem{proposition}[theorem]{Proposition}
    \newtheorem{corollary}[theorem]{Corollary}
\theoremstyle{definition}
    \newtheorem{definition}[theorem]{Definition}
    \newtheorem{notation}[theorem]{Notation}
    \newtheorem{conjecture}[theorem]{Conjecture}
    
\theoremstyle{remark}
    \newtheorem{remark}[theorem]{Remark}

\renewcommand{\ge}{\geqslant}
\renewcommand{\le}{\leqslant}

\newcommand{\stbt}[4]{\left(\begin{smallmatrix}#1 & #2 \\ #3 & #4\end{smallmatrix}\right)}

\begin{document}
\renewcommand{\urladdrname}{\itshape ORCID}

\title{Plectic structures in $p$-adic de Rham cohomology}
\author{David Loeffler}
\address[Loeffler]{Warwick Mathematics Institute, University of Warwick, Coventry CV4 7AL, UK}
\curraddr{UniDistance Suisse, Schinerstrasse 18, 3900 Brig, Switzerland}
\email{d.a.loeffler@warwick.ac.uk / david.loeffler@unidistance.ch}
\urladdr{\href{http://orcid.org/0000-0001-9069-1877}{0000-0001-9069-1877}}

\author{Sarah Livia Zerbes}
\address[Zerbes]{Department of Mathematics, ETH Z\"urich, R\"amistrasse 101, 8092 Z\"urich, Switzerland}
\email{sarah.zerbes@math.ethz.ch}
\urladdr{\href{http://orcid.org/0000-0001-8650-9622}{0000-0001-8650-9622}}

\thanks{D.L. gratefully acknowledges the support of the European Research Council through the Horizon 2020 Excellent Science programme (Consolidator Grant ``Shimura varieties and the BSD conjecture'', grant ID 101001051)}
\dedicatory{In memory of Jan Nekov\'a\v{r}}

\begin{abstract}
 Given a Hilbert modular form for a totally real field $F$, and a prime $p$ split completely in $F$, the $f$-eigenspace in $p$-adic de Rham cohomology has a family of partial filtrations and partial Frobenius maps, indexed by the primes of $F$ above $p$. The general plectic conjectures of Nekov\'a\v{r} and Scholl suggest a ``plectic comparison isomorphism'' comparing these structures to \'etale cohomology. We prove this conjecture in the case $[F : \QQ] = 2$ under some mild assumptions; and for general $F$ we prove a weaker statement which is strong evidence for the conjecture, showing that the plectic Hodge filtration has a canonical splitting given by intersecting with simultaneous eigenspaces for the partial Frobenii.
\end{abstract}

\maketitle

\section{Setup}

 Let $F$ be a totally real field of degree $d$, and $Y$ the Hilbert modular variety for $F$ of level $U_1(\mathfrak{N}) = \{ g \in \GL_2(\widehat{\cO}_F): g = \stbt{\star}{\star}{0}{1} \bmod \mathfrak{N}\}$. Fix a numbering of the embeddings $F \into \RR$ as $\sigma_1, \dots, \sigma_d$. Let $f$ be a newform of level $\mathfrak{N}$ and some weight $(\underline{k} + 2, \underline{t})$, where $\underline{k} = (k_1, \dots, k_d) \in \ZZ_{\ge 0}^d$, and $\underline{t} \in \ZZ^d$ such that $w = k_i + 2t_i$ is independent of $i$.

 We choose a prime $p$ which splits completely in $F$ and such that $(p, \mathfrak{N}) = 1$, and an isomorphism $\overline{\QQ}_p \cong \CC$, so we can denote the primes above $p$ by $\pp_1,  \dots, \pp_d$ with $\pp_i$ corresponding to $\sigma_i$. Finally, we also fix a finite extension $L$ of $\Qp$ containing the images of the Hecke eigenvalues of $f$.

 \subsection{The spaces \texorpdfstring{$D_p(f)$ and $V_p(f)$}{Dp(f) and Vp(f)}}

  We are interested in the $2^d$-dimensional de Rham cohomology eigenspace
  \[ D_p(f) \coloneqq H^d_{\dR}\left(Y_L, \cL_{\mu, \dR}\right)\{f\}, \]
  where $\cL_{\mu, \dR}$ is the vector bundle with connection determined by the weight $\mu = (\underline{k}, \underline{t})$, and $\{f\}$ denotes the $f$-generalised eigenspace for the Hecke operators away from $p\mathfrak{N}$. This is an $L$-vector space of dimension $2^d$, equipped with a Hodge filtration $\Fil^\bullet$, and (via comparison with crystalline cohomology) an $L$-linear Frobenius $\varphi$.

  We also have a representation of the Galois group $\Gamma_{\QQ} = \Gal(\overline{\QQ}/\QQ)$ given by
  \[ V_p(f) \coloneqq H^{d}_{\et}\left(Y_{\overline{\QQ}}, \cL_{\mu, \et}\right)\{f\}, \]
  where $\cL_{\mu, \et}$ is the \'etale local system of $L$-vector spaces corresponding to $\mu$. (We shall sometimes write just $D_p$ or $V_p$ for $D_p(f)$ or $V_p(f)$.) The Faltings--Tsuji comparison theorem of $p$-adic Hodge theory gives a canonical isomorphism of filtered $\varphi$-modules
  \begin{equation} \label{eq:faltings} D_p(f) \cong \DdR\left(V_p(f) |_{\Gamma_{\Qp}}\right). \end{equation}

 \subsection{Plectic structures}\label{sect:plecticstr}

  The ``plectic conjectures'' of \cite{NS-plectic1} predict that Shimura varieties for groups arising by restriction of scalars from a totally real field $F$, such as our $Y$, should carry canonical extra structures reflecting the arithmetic of $F$. Our goal here is to investigate how the isomorphism \eqref{eq:faltings} interacts with certain of these additional structures, as we now describe.

  \subsubsection*{Partial filtrations}

   The space $D_p(f)$ is endowed with a family of $d$ distinct (decreasing, $\ZZ$-indexed) filtrations $\Fil_i^\bullet D_p$, for $i = 1, \dots, d$, which we shall call \emph{partial filtrations}. Each filtration $\Fil_i^\bullet$ has two graded pieces in degrees $(t_i, t_i + k_i + 1)$; and the associated total filtration $\Fil^n$ defined by $\Fil^n D_p(f) = \sum_{\substack{(n_1, \dots, n_d) \in \ZZ^d \\ n_1 + \dots + n_d = n}} \Fil^{n_i}_i D_p(f)$ is the usual Hodge filtration.

    Roughly, the nontrivial subspace in $\Fil_i^\bullet$ corresponds to the part of the cohomology generated by differential forms which are holomorphic at the $i$-th infinite place. The construction is explained in \cite{nekovarscholl16} for $\underline{k} = (0, \dots, 0)$; the extension to general coefficients is routine, but we review it in \cref{sect:BGGcplx} below to fix notations.

  \subsubsection*{Partial Frobenii} We also have $d$ commuting linear maps $\varphi_1, \dots, \varphi_d$ on $D_p(f)$, the \emph{partial Frobenii}, whose composite is the usual Frobenius $\varphi$. These arise from endomorphisms of the special fibre $Y_0$, sending a Hilbert--Blumenthal abelian variety $A$ to the quotient of $A$ by the $\pp_i$-torsion in the kernel of Frobenius on $A$. We refer to  \cite[\S 4.6]{tianxiao16} for a detailed account of this construction.

 \subsection{Tensor induction}

  Attached to $f$, we also have a 2-dimensional standard Galois representation $V_p^{\mathrm{std}}(f)$ of $\Gamma_F \coloneqq \Gal(\overline{F} / F)$. Via results of \cite{brylinskilabesse84} and \cite{nekovar-semisimplicity}, we may choose an isomorphism of $\Gamma_\QQ$-representations
  \begin{equation} \label{eq:defpsi}
   \psi : V_p(f) \cong (\bigotimes-\mathrm{Ind})(V_p^{\mathrm{std}}(f)).
  \end{equation}
  This will often, but not always, be unique up to scalars.

  From $\psi$ we obtain an isomorphism of filtered $\varphi$-modules
  \[ \psi_p: D_p(f) \cong \bigotimes_{i=1}^d D_{\pp_i}(f) \]
  where $D_{\pp_i}(f) = \DdR\left(V_p^{\mathrm{std}} |_{\Gamma_{F_{\pp_i}}}\right)$ denotes the filtered $\varphi$-module of the standard representation at the prime $\pp_i$ above $p$.

  The nontrivial graded pieces of $D_{\pp_i}(f)$ are in degrees $t_i$ and $t_i + k_i + 1$; and the ``partial Eichler--Shimura'' congruence relation proved in \cite[Appendix A]{nekovar-semisimplicity} shows that we have $(\varphi_i - \alpha_i)(\varphi_i - \beta_i) = 0$ on $D_p(f)$, for each $i$, where $\alpha_i$, $\beta_i$ are the roots of the Hecke polynomial at $\pp_i$. The roots of this polynomial are also the eigenvalues of $\varphi$ on $D_{\pp_i}(f)$. These facts strongly suggest the following conjecture:

  \begin{conjecture}[Plectic comparison conjecture]
   \label{main}
   For some choice of global isomorphism $\psi$ as above, and each $i = 1, \dots, d$, the isomorphism $\psi_p$ intertwines the partial Frobenius $\varphi_i$ on $D_p(f)$ with the operator $1 \otimes \dots \otimes 1 \otimes \varphi \otimes 1 \otimes \dots \otimes 1$ (with $\varphi$ in the $i$-th component) on $\bigotimes_i D_{\pp_i}$; and similarly intertwines the $i$-th partial filtration $\Fil^\bullet_i$ on $D_p$ with the filtration $D_{\pp_1} \otimes \dots \otimes D_{\pp_{i-1}} \otimes (\Fil^\bullet D_{\pp_i}) \otimes D_{\pp_{i + 1}} \otimes \dots \otimes D_{\pp_d}$.
  \end{conjecture}

 \subsection{Relation to the plectic conjectures}

  \label{rmk:plecticGal}
  It is conjectured in \cite{nekovarscholl16} that $V_p(f)$ has an intrinsically defined action of the \emph{plectic Galois group} $\Gamma_{\QQ}^{\plec} \supseteq \Gamma_\QQ$, isomorphic to $S_d \ltimes \Gamma_F^d$, and that $\psi_p$ can be chosen to intertwine this with the obvious action of $\Gamma_{\QQ}^{\plec}$ on the tensor induction.

  If we assume the existence of this canonical $\Gamma_\QQ^{\plec}$-action on $V_p$, then we can state our conjecture more intrinsically as follows (removing the isomorphism $\psi$ from the picture). By restriction we obtain an action of $\Gamma_{\Qp}^{\plec} \cong (\Gamma_{\Qp})^d$ on $V_p$, and we can define
  \[ \DdR^{\plec}\left(V_p(f) |_{\Gamma_{\Qp}^{\plec}}\right) \coloneqq \left[V_p(f) \otimes (\mathbf{B}_{\dR} \otimes \dots \otimes \mathbf{B}_{\dR})\right]^{\Gamma_{\Qp}^{\plec}}. \]
  This gives a vector space equipped with $d$ partial Frobenii and partial filtrations, whose underlying (usual) filtered $\varphi$-module is $\DdR\left(V_p(f)|_{\Gamma_{\Qp}}\right)$. Then the ``correct'' conjecture is that \eqref{eq:faltings} is in fact an isomorphism
  \[ \DdR^{\plec}\left(V_p(f) |_{\Gamma_{\Qp}^{\plec}}\right) \cong D_p(f) \]
  commuting with the partial Frobenii and filtrations. This explains the name we have given to our conjecture: we are seeking a plectic refinement of the comparison isomorphism of $p$-adic Hodge theory. (In the formulation of \cref{main} above, lacking a canonically defined $\Gamma^{\plec}_{\Qp}$-action on $V_p$, we have substituted the $\Gamma^{\plec}_{\Qp}$-action transported from the tensor induction via $\psi$.)

  \begin{remark}
   It would be interesting to extend the conjecture to the case when $p$ is not totally split in $F$. In this case we do not know how to define the functor $\DdR^{\plec}$ (or even what the target category of this functor should be). We understand that this problem will be treated in forthcoming work of Lukas Kofler.
  \end{remark}

\section{Results for general \texorpdfstring{$d$}{d}: statements}

 We shall prove the following theorems in this paper. We note that all of these would be immediate consequences of \cref{main}, but our proofs are unconditional, and our purpose is to derive evidence for the conjecture. For brevity, we shall write $\Fil_i^+ = \Fil_i^{(t_i + 1)}$ for the unique nontrivial step in the $i$-th filtration; and for $S \subseteq \{1, \dots, d\}$ we put $\Fil^S D_p(f) = \bigcap_{i \in S} \Fil_i^+$. We shall recall the construction of these subspaces in the next section; in particular we will show that $\dim \left(\Fil^S D_p\right) = 2^{d - |S|}$ for every $S$, and the collection of subspaces $(\Fil^S D_p)_{S \subseteq \{1, \dots, d\}}$ forms a (decreasing) \emph{$I$-filtration} on $D_p(f)$ in the sense of \cite{NS-plectic1}, where $I$ is the lattice of subsets of $\{1, \dots, d\}$.

 \begin{definition}
  For $i \in \{1, \dots, d\}$ and $\alpha_i$ a root of the Hecke polynomial at $\pp_i$, we say $\alpha_i$ has \emph{strictly small slope} if $v_p(\alpha_i) < k_i + t_i$.
 \end{definition}

 This is a special case of the notion of ``strictly small slope'' defined in \cite{boxerpilloni20} for general reductive groups. Note that $v_p(\alpha_i)$ is always in the interval $[t_i, t_i + k_i + 1]$, and the slopes of the two roots are symmetric about the midpoint of the interval; so there is always at least one strictly-small-slope root if $k_i > 1$.

 \begin{theorem}
  \label{thm:main}
  Let $S \subseteq \{1, \dots, d\}$. For each $i \in S$, assume that the Hecke polynomial at $\pp_i$ has distinct roots, and that one of these roots $\alpha_i$ has strictly small slope.

  Then the simultaneous eigenspace $\bigcap_{i \in S} D_p(f)^{\varphi_i = \alpha_i}$ has dimension $2^{d - |S|}$, and has zero intersection with the sum $\sum_{i \in S} \Fil_i^+$, so the projection map
  \[ \bigcap_{i \in S} D_p(f)^{\varphi_i = \alpha_i} \xrightarrow{\ \cong\ } \frac{D_p(f)}{\sum_{i \in S} \Fil_i^+}\]
  is an isomorphism. Moreover, this isomorphism is strictly compatible with the partial filtrations, in the following sense: for each $T \subseteq \{1, \dots, d\}$, this map restricts to a bijection
  \[ \left(\Fil^T D_p(f)\right)\cap \left(\bigcap_{i \in S} D_p(f)^{\varphi_i = \alpha_i}\right)
   \xrightarrow{\ \cong\ } \frac{\Fil^T D_p(f)}{\Fil^T D_p(f) \cap\left(\sum_{i \in S} \Fil_i^+ D_p(f)\right)}, \]
  with both sides being zero unless $S \cap T = \varnothing$.
 \end{theorem}

 \begin{remark}
  The condition that the Hecke polynomial have distinct roots at each $i \in S$ is not actually needed for the theorem, but handling the case of a multiple root complicates the proof; and it is known that this case can never occur if the Tate conjecture holds (by the same arguments as in \cite{colemanedixhoven98}, using the assumption that $p$ splits completely in $F$), so we shall not pursue it here.
 \end{remark}

 The proof of \cref{thm:main} will be given in \cref{sect:mainproof}. Using this theorem, we can prove \cref{main} under some mild assumptions when $[F : \QQ] = 2$, see \S\ref{sect:quadratic}. In general we have the following corollaries:

 \begin{corollary}\label{cor:phipreserves}
  Let $j \in \{1, \dots, d\}$ be given, and suppose that for all $i \ne j$, the Hecke polynomial of $f$ at $\pp_i$ has two distinct roots which both have strictly small slope. Then the subspace $\Fil_j^+ D_p$ is stable under the operators $\varphi_i$ for $i \ne j$.
 \end{corollary}

 \begin{proof}
  We apply the theorem with $S = \{1, \dots, d\} - \{j \}$ and each of the $2^{d-1}$ choices of roots of the Hecke polynomials, which all satisfy the conditions. This gives $2^{d-1}$ linearly independent lines in $\Fil_j^+$, all of which are stable under $\varphi_i$ for $i \ne j$; and the sum of these subspaces must be all of $\Fil_j^+$.
 \end{proof}

 \begin{corollary}
  Suppose that for all $i \in \{1, \dots, d\}$, the Hecke polynomial has distinct roots, and that one of these roots $\alpha_i$ has strictly small slope. For each $S \subseteq \{1, \dots, d\}$, define
  \[ X(S) \coloneqq  (\Fil^S D_p(f)) \cap \left(\bigcap_{i \notin S} D_p(f)^{(\varphi_i = \alpha_i)}\right). \]
  Then $X(S)$ is one-dimensional for all $S$, and we have
  \[ D_p(f) = \bigoplus_{S \subseteq \{1, \dots, d\}} X(S). \]
  Moreover, the subspaces $X(S)$ split both the Hodge filtration, and the filtration by $\varphi_i$-eigenspaces, in the sense that for every $i \in \{1, \dots, d\}$ we have
  \[ \Fil_i^+ D_p(f) = \bigoplus_{\substack{S \subseteq \{1, \dots, d\} \\ i \in S} } X(S),\qquad\text{and}\qquad D_p(f)^{(\varphi_i = \alpha_i)} = \bigoplus_{\substack{S \subseteq \{1, \dots, d\} \\ i \notin S} } X(S). \]
 \end{corollary}

 These last two formulae should be seen as $p$-adic counterparts of the plectic Hodge decomposition on the cohomology over $\CC$ described in \cite{nekovarscholl16}.

 \begin{proof}
  We shall prove the following claim: for all $S$ we have $\Fil^S D_p = \bigoplus_{T \supseteq S} X(T)$. By downward induction on $|S|$, we may suppose that the assertion is true for every strict superset of $S$ (which is vacuously satisfied if $S = \{1, \dots, d\}$).

  We firstly claim that
  \[ \sum_{T \supsetneq S} \Fil^T D_p = \bigoplus_{T \supsetneq S} X(T). \]
  The induction hypothesis shows that $\sum_{T \supsetneq S} \Fil^T D_p = \sum_{T \supsetneq S} X(T)$, so what we must prove is that the sum is direct; but this follows by dimension-counting, since we know that $\sum_{T \supsetneq S} \Fil^T D_p$ has dimension $2^{d-|S|} - 1$, and each $X(T)$ has dimension one.

  Having proved the claim, it remains to show that $X(S)$ maps isomorphically onto $\Gr^S D_p$. Since $X(S)$ and $\Gr^S D_p$ are both one-dimensional, this amounts to the assertion that $X(S)$ is not contained in $\sum_{T \supsetneq S} \Fil^T$. But $\sum_{T \supsetneq S} \Fil^T \subseteq \sum_{i \notin S} \Fil_i^+$, and the theorem shows that this space has zero intersection with $\bigcap_{i \in S} D_p^{(\varphi_i = \alpha_i)}$.

  Having proved the claim, setting $S = \varnothing$ we conclude that $D_p$ is the direct sum of the $X(T)$'s, and the remaining assertions of the theorem are obvious.
 \end{proof}

\section{The plectic filtration}\label{sect:BGGcplx}

 Since we do not know a good reference for the existence of partial filtrations on $D_p(f)$ when $\underline{k} \ne (0, \dots, 0)$, we give an account of the construction below. In \cite{NS-plectic1} this is proved for $\underline{k} = (0, \dots, 0)$ using Dolbeault cohomology over $\CC$, as a consequence of the stronger statement that the complex-analytic cohomology has a plectic Hodge structure. We shall outline a slightly different proof for general coefficients, as this foreshadows the $p$-adic computations we shall use to prove \cref{thm:main} later in the paper. In this section all algebraic varieties are over $L$, and we write simply $Y$ instead of $Y_L$ for brevity.

 \subsection{Preliminaries on \texorpdfstring{$I$}{I}-filtrations}
  \newcommand{\cA}{\mathcal{A}}

  We begin with some generalities on the theory of filtrations indexed by lattices, as developed in \cite{NS-plectic1}. Throughout this paper, $I$ will denote the set of subsets of $\{1, \dots, d\}$ (for some integer $d \ge 1$), with the lattice operations of union and intersection; this is a distributive lattice.

  \begin{lemma}
   Let $\cA$ be an abelian category, and $M$ an object of $\cA$. If we are given arbitrary subobjects $\Fil_i^+ M \subseteq M$ for $i = 1, \dots, d$, then the family of subsets $(\Fil^S M)_{S \in I}$ defined by $\Fil^S M = \bigcap_{i \in S} \Fil_i^+ M$ is a (decreasing) \emph{weak $I$-filtration} in the sense of \cite[(1.1.2)]{NS-plectic1}, and every decreasing weak $I$-filtration has this form.\qed
  \end{lemma}

  In general these weak $I$-filtrations will not be $I$-filtrations (in the sense of Definition 1.2.1 of \emph{op.cit.}); the prototypical example is to take $d = 3$, $M = K^2$ for a field $K$, and the subspaces  $\Fil_i^+ M$ to be any three distinct lines in $M$.

  \begin{proposition}
   \label{prop:distrib}
   A weak $I$-filtration $(\Fil^S M)_{S \in I}$ is an $I$-filtration if and only if the the \emph{distributivity} condition
   \[ \left(\Fil^S M\right) \cap \left(\Fil^T M + \Fil^U M\right) = \left(\Fil^{S \cap T} M\right) + \left(\Fil^{S \cap U} M\right) \]
   holds for all $S, T, U \subseteq \{1, \dots, d\}$.
  \end{proposition}

  \begin{proof}
   See Proposition 1.2.7 of \cite{poliposi05}.
  \end{proof}

 \subsection{BGG complexes}
  We choose a smooth projective toroidal compactification $\imath: Y \into X$. Recall that $\mu$ denotes the character $(\underline{k}, \underline{t})$ of the diagonal maximal torus of $G = \operatorname{Res}_{F / \QQ} \GL_2$.

 \begin{definition} \
  \begin{enumerate}[(i)]
   \item For $T \subseteq \{1, \dots, d\}$, let $\kappa_T$ denote the (possibly non-dominant) weight of the diagonal torus of $G$ whose $i$-th entry is $(k_i + 2, t_i)$ if $i \in T$, and $(-k_i, t_i + k_i + 1)$ otherwise. We let $\omega_T$ denote the line bundle on $X$ corresponding to $\kappa_T$, so that $f$ itself is a global section of $\omega_{\{1, \dots, d\}}$.
   \item The (dual) \emph{BGG complex} of weight $\mu$ is the complex of sheaves
   \[ \BGG = \left[ \omega_\varnothing \to \bigoplus_{|T| = 1} \omega_T \to \bigoplus_{|T| = 2} \omega_T \to \dots \to \omega_{\{1, \dots, d\}}\right], \]
   with differentials as in \cite[\S 2.15]{tianxiao16}.
  \end{enumerate}
 \end{definition}

 By a result of Faltings (see theorem 2.16 of \cite{tianxiao16}), the BGG complex is known to be quasi-isomorphic to the pushforward to $X$ of the Rham complex on $Y$ with coefficients in the vector bundle $\cL_{\mu, \dR}$, i.e. $\imath_*\! \left(\cL_{\mu, \dR} \otimes \Omega^\bullet_Y\right)$. Thus $\HH^d(X, \BGG) \cong H^d_{\dR}(Y, \cL_{\mu, \dR})$, and in particular we have $D_p = \HH^d(X, \BGG)\{f\}$. The general theory of derived functors gives a \emph{hypercohomology spectral sequence} starting at the $E_1$ page, $E_1^{ij} = H^j(X, \mathcal{B}^i_\mu) \Rightarrow  \HH^d(X, \BGG)$. All terms in this spectral sequence (from the $E_1$ page onwards) are finite-dimensional vector spaces with an action of the Hecke operators (compatible with the action on the abutment).

 \begin{remark}
  Note that (for $F \ne \QQ$) the Hecke operators do not act as correspondences on any one specific choice of toroidal compactification; but the cohomology of automorphic vector bundles is independent of the choice of toroidal boundary data, and the direct limit over all such compactifications does carry a Hecke action. So the action of the Hecke algebra on the spectral sequence is well-defined.
 \end{remark}

 Since passing to a generalised eigenspace is an exact functor, we obtain a spectral sequence $E_1^{ij} = \bigoplus_{|T| = i} H^j(X, \omega_T)\{f\}$ converging to $D_p(f)$. This spectral sequence simplifies greatly, because of the following standard result on the cohomology of the sheaves $\omega_T$ (which follows from the computation of $\overline{\partial}$-cohomology of discrete-series representations of $\operatorname{SL}_2(\mathbb{R})$; see \cite[\S 8]{harris90b}):

 \begin{proposition}\label{prop:vanish}
  The $f$-generalised eigenspace $H^i(X, \omega_T)\{f\}$ is zero if $i \ne d - |T|$, and is one-dimensional otherwise.\qed
 \end{proposition}

 Hence the spectral sequence degenerates at $E_1$, and $D_p(f)$ has a filtration with graded pieces $\Gr^i = \bigoplus_{|T| = i} H^{d-i}(X, \omega_T)\{f\}$ for $0 \le i \le d$. Our aim is to ``upgrade'' this to an $I$-filtration. We shall do this by first defining an $I$-filtration on the \emph{complex} $\BGG$. We will need the following auxiliary definition:

 \begin{definition}
  For $S \subseteq \{1, \dots, d\}$, the direct sum of the $\omega_T$ with $T \supseteq S$ is a subcomplex, and we denote this by $\Fil^S \BGG$.
 \end{definition}
%

 \begin{proposition}
  The collection of subcomplexes $\left(\Fil^S \BGG\right)_{S \in I}$ forms a decreasing $I$-filtration of $\BGG$ (in the abelian category of complexes of abelian sheaves on $X$).
 \end{proposition}

 \begin{proof}
  By \cref{prop:distrib}, it suffices to verify the distributivity condition. Since the sum and intersection of subcomplexes is defined termwise, we may verify this for each degree of the complex individually (ignoring the differentials). As the filtration on each degree is given by a grading, the distributivity condition is obvious.
 \end{proof}
%
%

 \begin{proposition}
  \label{prop:hypercoh}
  For each $S \in I$, the $f$-generalised eigenspace in the hypercohomology of $\Fil^S \BGG$ vanishes outside degree $d$, and the natural map
  \[ \HH^d(X, \Fil^S \BGG)\{f\} \to \HH^d(X, \BGG)\{f\} = D_p(f)\]
  is an injection.
 \end{proposition}

 \begin{proof}
  The vanishing outside degree $d$ follows by applying the hypercohomology spectral sequence to $\Fil^S \BGG$ and using \cref{prop:vanish}, exactly as for the full complex $\BGG$. The same argument applied to the quotient complex $\BGG / \Fil^S$ shows that this also has vanishing hypercohomology outside degree $d$. Since we have a long exact sequence of hypercohomology
  \[ \dots \to \HH^{d-1}(X, \BGG / \Fil^S) \to \HH^d(X, \Fil^S \BGG) \to \HH^d(X, \BGG) \to \dots, \]
  this implies that $\HH^d(X, \Fil^S \BGG) \to \HH^d(X, \BGG)$ is injective on the $f$-generalised eigenspace.
 \end{proof}

 \begin{definition}
  For $S \in I$, we define $\Fil^S D_p(f)$ to be the image of the injection $\HH^d(X, \Fil^S \BGG)\{f\} \into D_p(f)$.
 \end{definition}

 \begin{proposition}\label{prop:plecfiltration}
  The collection of subspaces $\Fil^S D_p$ for $S \subseteq \{1, \dots, d\}$ forms a decreasing $I$-filtration of $D_p$. Moreover, for each $S \subseteq \{1, \dots, d\}$, the graded piece
  \[ \Gr^S D_p \coloneqq \Fil^S D_p / \Big(\sum_{T \supsetneq S} \Fil^T D_p\Big) \]
  is canonically isomorphic to $H^{d - |S|}(X, \omega_S)\{f\}$, and in particular is 1-dimensional.
 \end{proposition}

 \begin{proof}
  Let $\Lambda$ be the lattice of subcomplexes of $\BGG$ generated by the $\Fil^S \BGG$ under the operations of sum and intersection. From \cref{prop:distrib}, we know that $\Lambda$ is a distributive lattice. We claim that the map from $\Lambda$ to the lattice of subspaces of $D_p(f)$, given by mapping $\mathcal{C}$ to the image of $\HH^d(X, \mathcal{C})\{f\}$ in $D_p(f)$, respects the lattice operations of sum and intersection.

  Any object $\mathcal{C}$ of $\Lambda$ has the property that, for each $i$, the $i$-th term $\mathcal{C}^i$ is a direct sum of some collection of $\omega_U$'s with $|U| = i$. Hence, by the same argument as above, $\HH^i(X, \mathcal{C})\{f\}$ vanishes outside degree $d$, and its degree $d$ part injects into $D_p(f)$.

  So, if we consider two objects $\mathcal{C}$, $\mathcal{D}$ of $\Lambda$, and form the long exact hypercohomology sequence associated to the exact sequence of complexes
  \[ 0 \to \mathcal{C} \cap \mathcal{D} \to \mathcal{C} \oplus \mathcal{D} \to \mathcal{C} + \mathcal{D} \to 0, \]
  we obtain a short exact sequence of $\{f\}$-parts in degree $d$
  \[ 0 \to \HH^d(X, \mathcal{C} \cap \mathcal{D})\{f\} \to
  \HH^i(X, \mathcal{C})\{f\} \oplus  \HH^d(X, \mathcal{D})\{f\} \to  \HH^d(X, \mathcal{C} + \mathcal{D})\{f\} \to 0. \]
  The exactness at the middle and left terms gives compatibility with intersections, and the exactness at the right gives compatibility with sums, proving the claim.

  It follows that the lattice of subspaces of $D_p(f)$ generated by $\Fil^S D_p(f)$ is exactly the image of $\Lambda$ under a morphism of lattices; thus it is distributive, since $\Lambda$ is. By the converse direction of \cref{prop:distrib}, $\left(\Fil^S D_p(f)\right)_{S \in I}$ is an $I$-filtration.

  This also shows that the $S$-th graded piece of $D_p(f)$ is given by applying $\HH^d(X, -)\{f\}$ to $\Gr^S \BGG$. Since $\Gr^S \BGG$ is the single term $\omega_S$ in degree $|S|$, this is just $H^{d-|S|}(X, \omega_S)\{f\}$.
 \end{proof}

 Since all graded pieces of this $I$-filtration have dimension 1, we in particular deduce that for every $S$ we have $\dim \Fil^S D_p(f) = \# \{ T \in I : T \supseteq S\} = 2^{d - |S|}$. This completes the description of the $I$-filtration invoked in the previous section.

 \begin{remark}
  We have worked over the $p$-adic field $L$ since this is the setting of \cref{main}; but the arguments of this section are also valid if we replace $L$ with a number field $E$ containing the Galois closure $F^{\mathrm{Gal}}$ of $F$ and the Hecke eigenvalues of $f$, giving an $I$-filtration of $D_E(f) = H^d_{\dR}(Y_E, \cL_{\mu, \dR})\{f\}$ in the category of $E$-vector spaces. (Note that the assumption that $E$ contain $F^{\mathrm{Gal}}$ is necessary here, since the action of $F^{\mathrm{Gal}} / F$ does not preserve the individual summands in the BGG complex, but rather permutes them according to the action of the Galois group on $\{1, \dots, d\}$; so if the coefficients of $f$ lie in $\QQ$, then $D_{\QQ}(f)$ makes sense as $\ZZ$-filtered vector space over $\QQ$, but it only acquires an $I$-filtration after base-extension to $F^{\mathrm{Gal}}$.)
 \end{remark}

\section{Proof of \cref{thm:main}, I: geometric Jacquet-Langlands}
 \label{sect:mainproof}

 \subsection{The \texorpdfstring{$S$}{S}-ordinary locus} We now embark on the proof of the theorem. Since $p$ is coprime to $\mathfrak{N}$ and unramified in $F$, the varieties $X$ and $Y$ have canonical smooth models over $\mathcal{O}_L$ (compatible with the embedding $Y \into X$); we write $X_0$ and $Y_0$ for their special fibres.

  \begin{definition}
   Let $X^{S-\ord}_0$ denote non-vanishing locus of the partial Hasse invariants \cite[\S 3.2]{tianxiao16} for the primes $\pp_i$ with $i \in S$.
  \end{definition}

  This is the locus where the $\pp_i^\infty$-torsion of the universal semiabelian variety over $X_0$ is ordinary for all $i \in S$. We write $\cX^{S-\ord}$ for the tube of $X^{S-\ord}_0$ in the dagger analytification $\cX = X^{\mathrm{an}, \dagger}$, which is a dagger space over $L$.

 \subsection{The spectral sequence of Goren--Oort strata}

  The complement of $X_0^{S-\ord}$ in $X_0$ is a normal-crossing divisor $\bigcup_{i \in S} Z_i$, where $Z_i$ is the vanishing locus of the Hasse invariant at $\pp_i$. If we define $Z_T$, for each $T \subseteq S$, to be the intersection $\bigcap_{i \in T} Z_i$ (understood as $X_0$ if $T = \varnothing$), then each $Z_T$ is a smooth closed subvariety of codimension $|T|$ in $X_0$. These are (closures of) \emph{Goren--Oort strata} in $X_0$; see \cite[\S 1.3]{tianxiao16-GO}. We write $\mathcal{Z}_T$ for the tube of $Z_T$ in $\cX$.

  \begin{proposition}
   There is a spectral sequence
   \[ E_1^{ij} = \bigoplus_{T\subseteq S, |T| = i} \HH^j(\mathcal{Z}_T, \BGG) \Rightarrow \HH^{i + j}_c(\cX^{S-\ord}, \BGG).\]
  \end{proposition}

  \begin{proof}
   As in \cite[\S 5.2]{lestum07}, we have a left exact functor $\Gamma_{\cX^{S-\ord}}(-)$, ``sections with support in $\cX^{S-\ord}$'', on the category of abelian sheaves on $\cX$, and the derived functor $R\Gamma_{\cX^{S-\ord}}(-)$ fits into an exact triangle
   \[ R\Gamma_{\cX^{S-\ord}}(\mathcal{E}) \to \mathcal{E}\to \mathcal{E} |_{(\cX - \cX^{S-\ord})} \to [+1]\]
   for any abelian sheaf (or complex of sheaves) $\mathcal{E}$.

   We are interested in the object $R\Gamma_{\cX^{S-\ord}}\left(\BGG\right)$ of the derived category, whose hypercohomology is by definition $\HH^*_{c}\left(\cX^{S-\ord}, \BGG \right)$. We claim this object is isomorphic to the total complex of the double complex (which we denote by $C_S$)
   \[ \BGG \to \bigoplus_{\substack{T \subseteq S \\ |T| = 1}} \BGG|_{\mathcal{Z}_T} \to \bigoplus_{\substack{T \subseteq S \\ |T| = 2}}\BGG|_{\mathcal{Z}_T} \to \dots, \]
   where the morphisms are alternating sums of restriction maps. There is nothing to prove if $S = \varnothing$, so let us proceed by induction on $|S|$, and write $S = \{a\} \sqcup S'$ for some $a$. Then we have $\cX^{S-\ord} = \cX^{\{a\}-\ord} \cap \cX^{S'-\ord}$, so $R\Gamma_{\cX^{S-\ord}} = R\Gamma_{\cX^{\{a\}-\ord}} \circ R\Gamma_{\cX^{S'-\ord}}$. The exact for the functor $R\Gamma_{\cX^{\{a\}-\ord}}$ applied to the object $R\Gamma_{\mathcal{X}^{S'-\ord}}\left(\BGG\right)$ reads
   \[ R\Gamma_{\mathcal{X}^{S-\ord}} \left(\BGG\right) \to R\Gamma_{\mathcal{X}^{S'-\ord}} \left(\BGG\right)\to \left(R\Gamma_{\mathcal{X}^{S'-\ord}}\left(\BGG\right)\right)|_{\mathcal{Z}_{\{a\}}} \to [+1].\]
   Applying the induction hypothesis for $S'$, we obtain an identification of $R\Gamma_{\mathcal{X}^{S-\ord}} \left(\BGG\right)$ with the mapping fibre of $C_{S'} \to C_{S'} |_{\mathcal{Z}_{\{a\}}}$, which is easily seen to be isomorphic to $C_S$, proving the claim. Applying the hypercohomology functor now gives the claimed spectral sequence.
  \end{proof}

  \begin{proposition}\label{prop:rigcoh-iso}
   The natural pushforward map
   \[ \HH^d_{c}\left(\cX^{S-\ord}, \BGG \right)\{f\} \to \HH^d\left(\cX, \BGG \right)\{f\} = D_p \]
   is a bijection.
  \end{proposition}

  \begin{proof}
   For $T \ne \varnothing$, the variety $Z_T$ is a smooth closed subvariety of $X_0$ contained in $Y_0$; so the restriction of $\BGG$ to $\mathcal{Z}_T$ is the de Rham complex of $\cL_{\mu, \dR}$, and hence $\HH^i\left(\mathcal{Z}_T, \BGG\right)$ is the \emph{rigid cohomology} of $Z_T$ with coefficients in the $F$-isocrystal $\cL_{\mu, \mathrm{rig}}$ corresponding to $\cL_{\mu, \mathrm{dR}}$, cf.~\cite[\S 4.5]{tianxiao16}. (This can be extended to $T = \varnothing$ if we interpret $\cL_{\mu, \mathrm{rig}}$ as a `log $F$-isocrystal', to account for the logarithmic singularity of the connection along the boundary $X - Y$; but this is not essential for our arguments, since for $T = \varnothing$ we already have an interpretation of $\HH^i\left(\mathcal{Z}_T, \BGG\right)$ as the de Rham cohomology of the algebraic variety $Y_L$.)

   By general properties of rigid cohomology, each term $E_1^{ij}$ is finite-dimensional, and the entire spectral sequence has a natural action of the prime-to-$p\mathfrak{N}$ Hecke algebra; so we may pass to $f$-generalised eigenspaces to obtain a spectral sequence
   \[ \bigoplus_{T \subseteq S, |T| = i} H^j_{\mathrm{rig}}\left(Z_{T}, \cL_{\mu, \mathrm{rig}}\right)\{f\} \Rightarrow \HH^{i+j}_{c}\left(\cX^{S-\ord}, \BGG\right)\{f\}. \tag{$\star$}\]

   The main result of \cite{tianxiao16-GO} (Theorem 5.2 of \emph{op.cit.}; see also \cite[\S 5.21]{tianxiao16}) is that the $Z_T$ are themselves Shimura varieties: after base-extending to $\overline{\FF}_p$, we can identify $Z_T$ with the special fibre of the Shimura variety associated to the quaternion algebra $B_T$ over $F$ ramified at the finite places $\pp_i$ and the infinite places $\sigma_i$, for $i \in T$. Moreover, this identification is compatible with Hecke correspondences away from $p$. Using part (2) of \cite[Theorem 5.8]{tianxiao16-GO}, one can also check that the restriction of $\cL_{\mu, \mathrm{rig}}$ to $Z_T$ is the $F$-isocrystal associated to the weight $\mu$ algebraic representation of $B_T^\times$.

   The rigid cohomology of each $Z_T$ can therefore be computed (as a module for the prime-to-$p$ Hecke algebra) in terms of automorphic representations of $B_T^\times$ of level $\mathfrak{N}$ and weight $\mu$. These are precisely the Jacquet--Langlands transfers to $B_T^\times$ of Hilbert modular forms for $\GL_2 / F$ which are new of level $\mathfrak{N}\cdot \prod_{i \in T} \pp_i$. Since $f$ is new of level $\mathfrak{N}$, it follows that for $T \ne \varnothing$, the $f$-generalised eigenspaces in these cohomology groups are zero. Thus the $E_1^{ij}$ terms in the spectral sequence $(\star)$ are zero for $i \ne 0$, and hence the edge maps in the $i = 0$ column are isomorphisms (for all $j$, and in particular for $j = d$).
  \end{proof}

  \begin{remark}
   This is essentially the same argument as the main theorem of \cite{tianxiao16}. More precisely our setting is the ``Poincar\'e dual'' of theirs: in \emph{op.cit.}, they compute the rigid cohomology of $\cX^{S-\ord}$ (for $S$ the whole set $\{1, \dots, d\}$) with compact support towards the toroidal boundary divisor $X - Y$ but non-compact support towards the $Z_i$, while we compute the cohomology with compact support towards the $Z_i$ and non-compact support towards the toroidal boundary.
  \end{remark}

  \subsection{A new filtration} From the decomposition of the BGG complex we obtain a new filtration on $D_p$:

  \begin{notation}
   For $T \subseteq \{1, \dots, d\}$, define $\sF^T D_p$ to be the $f$-generalised eigenspace in the space
    \[ \operatorname{image}\Big(\HH^d_{c}\left(\cX^{S-\ord}, \Fil^T \BGG \right) \to
   \HH^d_{c}\left(\cX^{S-\ord}, \BGG \right)\Big). \]
  \end{notation}

  Note that \emph{a priori} the subspaces $\sF^T$ only define an ``$I$-prefiltration'' in the sense of \cite{NS-plectic1}; it is not clear if it is a $\ZZ^d$-filtration (or even a weak $\ZZ^d$-filtration). Moreover, it is clear that there are inclusions
  \[ \sF^T D_p \subseteq \Fil^T D_p, \]
  since the composite map $\HH^i_c(\cX^{S-\ord}, \Fil^T \BGG) \to \HH^i_c(\cX^{S-\ord}, \BGG) \to  \HH^i(\cX, \BGG)$ factors through $\HH^i(\cX, \Fil^T \BGG)$; but it is far from obvious \emph{a priori} if equality holds. On the other hand, we have a new piece of information: the partial Frobenii $\varphi_i$ for $i \in S$ admit liftings to $\cX^{S-\ord}$, and the coefficient sheaves are compatible with these liftings. Hence, for each $T \in \{1, \dots, d\}$, the subspace $\sF^T D_p(f)$ is invariant under the $\varphi_i$ for $i \in S$ (while we have no reason to expect the $\Fil^T D_p(f)$ to be invariant under these operators).

\section{Proof of \cref{thm:main}, II: higher Coleman theory}

  \subsection{Higher Coleman theory: statements}

   In order to better understand the prefiltration $(\mathscr{F}^T D_p)_{T \in I}$, we now use methods from the \emph{higher Coleman theory} introduced in \cite{boxerpilloni20} to study the groups $H^*_c(\cX^{S-\ord}, \omega_T)$. We begin by stating the two key results we need, whose proofs we shall explain in the remainder of this section.

  \begin{proposition}
   \label{prop:compact}
   For each $T \subseteq \{1, \dots, d\}$ and each $n$, the operator $\varphi_S = \prod_{i \in S} \varphi_i$ on $H^n_c\left(\cX^{S-\ord}, \omega_T \right)$ is potentially compact (i.e.~there is some $k \ge 1$ such that $\varphi_S^k$ is compact).
  \end{proposition}

  It follows that the subspace $H^n_c\left(\cX^{S-\ord},\omega_T\right)\{(\alpha_i)_{i \in S}\}$, defined as the maximal subspace on which the operators $\varphi_i - \alpha_i$ for $i \in S$ are all nilpotent, is finite-dimensional. So we may in particular decompose it as a direct sum of generalised eigenspaces for the prime-to-$p\mathfrak{N}$ Hecke operators; we write $H^n_c(\dots)\{(\alpha_i)_{i \in S}, f\}$ for the summand corresponding to $f$.

  \begin{proposition} \label{prop:classicity} Let $T \subseteq \{1, \dots, d\}$. Then:
   \begin{itemize}
    \item If $S \cap T \ne \varnothing$, then $H^n_c\left(\cX^{S-\ord},\omega_T\right)\{(\alpha_i)_{i \in S}, f\}  = 0$ for all $n$.
    \item If $S \cap T = \varnothing$, then $H^n_c\left(\cX^{S-\ord},\omega_T\right)\{(\alpha_i)_{i \in S}, f\} = 0$ for $n \ne d - |T|$.
    \item If $S \cap T = \varnothing$ and $n = d - |T|$, then $H^n_c\left(\cX^{S-\ord},\omega_T\right)\{(\alpha_i)_{i \in S}, f\}$ is 1-dimensional, and the natural ``forget supports'' map is an isomorphism
    \[
    H^{d - |T|}_c\left(\cX^{S-\ord}, \omega_T\right)\{(\alpha_i)_{i \in S}, f\}
    \xrightarrow{\ \cong\ }
    H^{d - |T|}\left(\cX, \omega_T\right)\{f\}.
    \]
   \end{itemize}
  \end{proposition}

  Note that the target of the map in the last bullet point does \textit{not} have a natural action of the partial Frobenii.

 \subsection{Reformulation at parahoric level}

  \begin{definition}
   We let $Y_S \to Y$ be the Shimura variety of level $\{ g \in U_1(\mathfrak{N}) : g = \stbt{\star}{\star}{0}{\star} \bmod \pp_i\, \forall i \in S\}$.
  \end{definition}

  This is a Shimura variety of parahoric level at $p$, so it has a canonical regular $\Zp$-model; this is the moduli space classifying choices of cyclic $p$-subgroup $C_i \subseteq A[\pp_i]$ for each $i \in S$, where $A / Y$ is the universal Hilbert--Blumenthal abelian variety. By a suitable choice of the toroidal boundary data, we can (and do) assume that the natural map $Y_S \to Y$ extends to a map of toroidal compactifications $X_S \to X$, where $X_S$ is projective, and smooth in a neighbourhood of the cusps; and the $C_i$ extend to finite flat group schemes over $X_S$, via Mumford's construction.

  \begin{definition}
   For $T \subseteq S$, we let $X_{S, 0}^{T-\mul}$ denote the open subvariety of the special fibre $X_{S, 0}$ where the level subgroups $C_i$ are of multiplicative type for $i \in T$, and \'etale for $i \in S - T$. We write $\cX_S^{T-\mul}$ for the tube of $X_{S, 0}^{T-\mul}$ in the dagger space $\cX_S$.
  \end{definition}

  These subspaces are disjoint, and their union is the $S$-ordinary locus. The fully multiplicative subspace $\cX_S^{S-\mul}$ maps isomorphically to $\cX^{S-\ord}$, the $S$-ordinary locus at prime-to-$p$ level,  since over $\cX^{S-\ord}$, the $\pp_i$-torsion of $A$ has a unique multiplicative $p$-subgroup (the canonical subgroup).

  \begin{definition}
   Let $U_{\pp_i}'$, for each $i \in S$, be the Hecke correspondence given by the double coset of $\stbt{1}{}{}{\varpi_i} \in \GL_2(F_{\pp_i})$, where $\varpi_i$ is a uniformizer.
  \end{definition}

  \begin{remark}
   Note that $U'_{\pp_i}$ is not quite the same as the more familiar operator $U_{\pp_i}$ defined by the double coset of $\stbt{\varpi_i}{}{}{1} \in \GL_2(F_{\pp_i})$; it is $U_{\pp_i}$, not $U'_{\pp_i}$, which has a straightforward formula in terms of $q$-expansions at $\infty$. As operators on the coherent cohomology of the algebraic variety $X_{/L}$, we can describe $U'_{\pp_i}$ as the transpose of $U_{\pp_i}$ with respect to Serre duality; however, we cannot use this as a \emph{definition} of $U'_{\pp_i}$, since we shall shortly need to consider the action of this operator on the cohomology of the dagger spaces $\cX_S^{S-\mul}$, and we do not know if Serre duality holds for the cohomology of non-affinoid dagger spaces.
  \end{remark}

  \begin{lemma}
   The restriction of $U_{\pp_i}'$ to $\cX_S^{S-\mul}$ coincides, under the isomorphism $\cX_S^{S-\mul} \cong \cX^{S-\ord}$, with our partial Frobenius lifting $\varphi_i$.
  \end{lemma}

  \begin{proof}
   In moduli-theoretic terms, the action of $U_{\pp_i}'$ is given by quotienting the abelian variety $A$ by the level subgroup $C_i$, and summing over possible choices of cyclic $p$-subgroups $C_i' \subseteq (A / C_i)[\pp_i]$. If $A$ is ordinary at $\pp_i$, and $C_i = \hat{A}[\pp_i]$ (the $\pp_i$-torsion of the formal group of $A$) is the unique multiplicative subgroup, then exactly one of these subgroups $C'_i$ is multiplicative (the image of $\hat{A}[\pp_i^2]$), and the remaining $p$ subgroups are \'etale. Thus the restriction of $U_{\pp_i}'$ to the $S$-multiplicative locus is actually a morphism (not just a correspondence). Moreover, since $\hat{A}[\pp_i]$ is the $\pp_i$-part of the kernel of Frobenius, we conclude that the restriction of $U_{\pp_i}'$ coincides, under our identification $\cX_S^{S-\mathrm{mul}} \cong \cX^{S-\ord}$, with the partial Frobenius $\varphi_i$.
  \end{proof}

  By the functoriality of pushforward maps, for any $T \subseteq \{1, \dots, d\}$, we have a commutative diagram (compatible with the action of Hecke operators away from $p$)
  \[\begin{tikzcd}
   H^*_c(\cX_S^{S-\mul}, \omega_T) \rar \dar[equals] &
   H^*(\cX_S, \omega_T) \dar\\
   H^*_c(\cX^{S-\ord}, \omega_T) \rar&
   H^*(\cX, \omega_T).
  \end{tikzcd}\]
  Moreover, the cohomology of $\cX_S$ is isomorphic (by the GAGA theorem) to the algebraic de Rham cohomology of the variety $X_S$, which can be computed using automorphic representations, as in \cref{prop:vanish} above. So the generalised eigenspace $H^n(\cX_S, \omega_T)\{(\alpha_i)_{i \in S}, f\}$ on which the prime-to-$p$ Hecke operators act via the Hecke eigenvalues of $f$, and $U_{\pp_i}'$ acts as $\alpha_i$ for each $i \in S$, is concentrated in degree $n = d - |T|$; and in this degree it is a 1-dimensional space and maps isomorphically to $H^{d-|T|}(\cX, \omega_T)\{f\}$. So to prove \cref{prop:compact,prop:classicity}, it suffices to prove the following:

  \begin{proposition}\label{prop:classicity2} Let $T \subseteq \{1, \dots, d\}$.
   \begin{enumerate}
    \item The operator $U'_S = \prod_{i \in S} U'_{\pp_i}$ is potentially compact on $H^n_c\left(\cX_S^{S-\mul}, \omega_T \right)$;
    \item if $S \cap T \ne \varnothing$, then $H^*_c\left(\cX_S^{S-\mul}, \omega_T \right)\{(\alpha_i)_{i \in S}\}$ vanishes in all degrees;
    \item if $S \cap T = \varnothing$, then the map
    \[ H^*_c\left(\cX_S^{S-\mul}, \omega_T \right)\{(\alpha_i)_{i \in S}\} \to
    H^*_c\left(\cX_S, \omega_T \right)\{(\alpha_i)_{i \in S}\}\]
    is an isomorphism in all degrees.
   \end{enumerate}
  \end{proposition}

 \subsection{The extremal case}

  In the case $S = \{1, \dots, d\}$, the statements of \cref{prop:classicity2} are instances of the theorems of \cite{boxerpilloni20}, applied to the reductive group $G = \operatorname{Res}_{F / \QQ}\GL_2$.

  \begin{itemize}
   \item Part (i) follows from \cite[Theorem 5.13 (2)]{boxerpilloni20}. In the notation of \emph{op.cit.}, we take the Kostant representative $w$ to be the identity, and the open subset $\mathcal{U}$ of the Shimura variety to be the entire space; hence the cohomology with support $R\Gamma_{\mathcal{Z} \cap \mathcal{U}}(\mathcal{U}, -)$ appearing \emph{loc.cit.} is the compactly-supported cohomology of $\mathcal{Z}$, which is an arbitrarily small rigid-analytic neighbourhood of the $S$-multiplicative locus $\mathcal{X}^{S-\mul}_S$. Since the compactly-suppored cohomology of the dagger space $\mathcal{X}^{S-\mul}_S$ is the direct limit of the cohomology of its rigid-analytic neighbourhoods, this gives the claim.

   \item Part (ii) is an instance of \cite[Corollary 5.65]{boxerpilloni20}. Since $\kappa_T$ has no component equal to 1, the set $C(\kappa)^-$ for $\kappa = \kappa_T$ is a singleton; explicitly, it is the element of the Weyl group $W_G = \prod_{i = 1}^d C_2$ whose $i$-th component is nontrivial iff $i \in T$. In particular, $\mathrm{id} \in C(\kappa_T)^-$ iff $T = \varnothing$. However, we have seen in the previous paragraph that the finite-slope part of $R\Gamma(\cX_S^{S-\mul}, \kappa_T)$ coincides with Boxer and Pilloni's $R\Gamma_w(K^p, \kappa, \chi)^{-,\mathrm{fs}}$ for $\kappa = \kappa_T$, $w = \mathrm{id}$, and suitable choices of $K^p$ and $\chi$. So the strictly small slope part of $R\Gamma(\cX_S^{S-\mul}, \kappa_T)$ vanishes for $T \ne \varnothing$; and this translates to the bounds for the valuations of $(\alpha_i)_{i = 1, \dots, d}$ which we have imposed here.

   \item Part (iii) follows in exactly the same way from \cite[Theorem 5.66]{boxerpilloni20}.
  \end{itemize}

 \subsection{Modifications for general \texorpdfstring{$S$}{S}}

  For a general subset $S \subseteq \{1, \dots, d\}$, we need to modify the theory developed in \cite{boxerpilloni20} slightly. The starting point for the constructions of \emph{op.cit.} is Scholze's Hodge--Tate period map (see Theorem 4.68 of \emph{op.cit.}), which is a continuous map from the perfectoid Shimura variety of infinite level at $p$ to a flag variety $\mathcal{FL}_{G, \mu}$. In our case, this flag variety is a product of copies of $\mathbf{P}^1$, indexed by $\{1, \dots, d\}$.

  If we compose this period map with the projection onto only those factors in the product given by $S$, we obtain a ``partial period map'', which only detects the level structure and Hodge filtration at a subset of the primes above $p$. In particular, the preimage of the $\Qp$-points of $(\mathbf{P}^1)^S$ is the $S$-ordinary locus (rather than the fully ordinary locus); and the subsets $\cX_S^{T-\mul}$, for $T \subseteq S$, are the preimages of the special points in the partial flag variety whose $i$-th component is $\infty$ for $i \in T$ and $0$ otherwise (the image of the partial Weyl group $W_{G, S} = \prod_{i \in S} C_2$).

  We can now run the entire machine of \emph{op.cit.} in this setting, using the partial period map to define loci in the Shimura variety $X_S$ which are the support conditions for cohomology groups, and to analyse the action of Hecke operators on these loci. This gives a spectral sequence (which for $S = \{1, \dots, d\}$ is the Bruhat-stratification sequence of \cite[Theorem 5.15]{boxerpilloni20}):
  \[ E_1^{ij} = \bigoplus_{\substack{U \subseteq S \\ |S \setminus U| = i}} H^{i + j}_{(U)}(\cX_S^{U-\mul}, \omega_T)^{-,\mathrm{fs}} \Longrightarrow H^{i+j}(\cX_S, \omega_T)^{-, \mathrm{fs}}, \]
  where ``$-,\mathrm{fs}$'' denotes the finite-slope part for $U'_S$ (which acts compactly on all the terms), and $(U)$ denotes an appropriate partial compact support condition depending on $U$ (fully compact support when $U = S$, and non-compact support when $U = \varnothing$). In particular, the $i = 0$ terms are the compactly-supported cohomology of $\cX_S^{S-\mul}$. Exactly as in the case $S = \{1, \dots, d\}$ treated in \emph{op.cit.}, one obtains lower bounds on the slopes of the $U'_{\pp_i}$-operators, and these imply that for a strictly-small-slope eigenvalue system $(\alpha_i)_{i \in S}$, the corresponding generalised eigenspace in $H^{i + j}_{(U)}(\cX_S^{U-\mul}, \omega_T)^{-,\mathrm{fs}}$ can only be non-zero when $U = S \cap T^c$, giving the proof of \cref{prop:classicity2}.

\section{Proof of \cref{thm:main}, III: conclusion}

  \begin{notation}
   For $S \subseteq \{1, \dots, d\}$, write
   \[ \Fil_S^- D_p = \frac{D_p}{\sum_{i \in S} \Fil_i^+ D_p},\]
   whose dimension is $2^{d - |S|}$. For $T \subseteq \{1, \dots, d\}$, we define $\Fil^T \Fil_S^- D_p$ as the image of $\Fil^T D_p$ in $\Fil_S^-$; note that this is zero unless $S \cap T = \varnothing$. One checks that this defines an $I$-filtration on $\Fil_S^- D_p$, where $I$ is the poset of subsets of $\{1, \dots, d\}$ as usual; and the $T$-th graded piece $\Gr^T\Fil_S^- D_p$ is isomorphic to $\Gr^{T} D_p$ if $S \cap T = \varnothing$, and is zero otherwise. (This is a slight variation on Proposition 1.3.7 of \cite{NS-plectic1}.)
  \end{notation}

  Since $H^n_c\left(\cX^{S-\ord},\omega_T\right)\{(\alpha_i)_{i \in S}, f\}$ vanishes outside degree $d - |T|$, we can argue exactly as in \cref{prop:plecfiltration} to see that the subspaces $\sF^T D_p(f)\{(\alpha_i)_{i \in S}\}$ satisfy the distributive property, and the $T$-th graded piece is $0$ if $S \cap T \ne \varnothing$ and maps isomorphically to the corresponding graded piece of $\Fil^\bullet D_p$ otherwise. This shows that $(\sF^T)_{T \in I}$ defines an $I$-filtration on the subspace $D_p^S \coloneqq \bigcap_{i \in S} D_p^{\varphi_i = \alpha_i}$, and that the map of $I$-filtered vector spaces
  \[ \left(D_p^S, \sF^\bullet\right) \to \left(\Fil_S^- D_p, \Fil^\bullet\right)\]
  induces an isomorphism on every graded piece, and is hence an isomorphism of $I$-filtered vector spaces.

  It remains to check that for each $T$ we have
  \[ D_p^S \cap \sF^T D_p = D_p^S \cap \Fil^T D_p.\]
  The inclusion ``$\subseteq$'' is clear, since $\sF^T D_p \subseteq \Fil^T D_p$. However, we know that the map
  \[ D_p^S / \left(D_p^S \cap \sF^T\right) \to \frac{\Fil_S^- D_p}{\Fil^T \Fil_S^- D_p}\]
  is a bijection; as this map clearly factors through $D_p^S / \left(D_p^S \cap \Fil^T\right)$, we conclude that $D_p^S \cap \Fil^T$ cannot be strictly larger than $D_p^S \cap \sF^T$, as this would contradict the injectivity of this map. This completes the proof of \cref{thm:main}.

\section{Quadratic case}
 \label{sect:quadratic}

 Throughout this section we suppose $[F : \QQ] = 2$. Then we can prove a stronger result by making use of the self-duality of $D_p$.

 \begin{theorem}\label{thm:nonBC}
  Suppose the following conditions hold:
  \begin{enumerate}
   \item[(a)] The character of $f$ is trivial.
   \item[(b)] There exists an $i \in \{1, 2\}$ such that at least one of the two roots $\alpha_i, \beta_i$ of the Hecke polynomial of $f$ has strictly small slope.
   \item[(c)] The set of pairwise products $\{ \alpha_1 \alpha_2, \alpha_1 \beta_2, \beta_1 \alpha_2, \beta_1 \beta_2\}$ has four distinct elements.
  \end{enumerate}
  Then \cref{main} is true.
 \end{theorem}

 We also have a complementary result for base-change forms (which never satisfy condition (c)).

 \begin{theorem}\label{thm:BC}
  Suppose the following conditions hold:
  \begin{enumerate}
   \item[(a')] $f$ is the base-change of an elliptic modular form $f_0$ with trivial character (so in particular $k_1 = k_2$).
   \item[(b')] The roots $\alpha_0, \beta_0$ of the Hecke polynomial of $f$ at $p$ are distinct, and at least one has strictly small slope.
   \item[(c')] We have $\alpha_0 / \beta_0 \notin \{\pm 1\}$.
  \end{enumerate}
  Then \cref{main} is true.
 \end{theorem}

 \begin{remark}
  The proof of these statements relies on the ``accident'' that the image of the tensor-product map $\GL_2 \times \GL_2 \to \GL_4$ has a nice description: it is the orthogonal similitude group $\operatorname{GSO}_4$. We do not know of a nice description of the image of the analogous map $\GL_2 \times \dots \times \GL_2 \to \GL_{2^d}$ for $d \ge 3$, so the methods of this section seem unlikely to generalise beyond the quadratic case.
 \end{remark}

\subsection{Proof of \cref{thm:nonBC}}

 For any $F$ we have an isomorphism of \'etale sheaves $\cL_{\mu, \et}^\vee \cong \cL_{\mu, \et}(dw)$ so we obtain a perfect Poincar\'e duality pairing
 \[ H^d_{\et, !}(Y_{\overline{\QQ}}, \cL_{\mu, \et}) \times H^d_{\et, !}(Y_{\overline{\QQ}}, \cL_{\mu, \et}) \to H^{2d}_{\et, c}(Y_{\overline{\QQ}}, \Qp(-dw)) = \Qp(-d(w + 1)), \]
 where $!$ denotes interior cohomology. The transpose of the Hecke operator $T_\mathfrak{q}$ for an unramified prime $\mathfrak{q}$ is $\langle \mathfrak{q}\rangle^{-1} T_\mathfrak{q}$, where $\langle \mathfrak{q}\rangle^{-1}$ is the diamond operator; so if $f$ has trivial character, we obtain a perfect pairing on the $f$-generalised eigenspace. Moreover, since this pairing is given by a cup-product in degree $d$ (and the cup-product is graded-commutative), it is a symmetric bilinear form if $d$ is even and antisymmetric if $d$ is odd. Thus, in the $d = 2$ case we obtain a canonical symmetric bilinear form on $V_p(f)$, which we denote by $\lambda$. A similar construction using Poincar\'e duality for de Rham cohomology gives a symmetric bilinear form on $D_p(f)$, which we also denote by $\lambda$; and these two bilinear forms are compatible via the functor $\DD_{\dR}$.

 We also have a canonical-up-to-scalars symmetric bilinear form on the tensor induction $(\bigotimes-\mathrm{Ind})(V_p^{\mathrm{std}}(f))$, arising from the symplectic self-duality of $V_p^{\mathrm{std}}(f)$: the underlying space of $(\bigotimes-\mathrm{Ind})(V_p^{\mathrm{std}}(f))$ can be identified with $V_p^{\mathrm{std}}(f) \otimes V_p^{\mathrm{std}}(f)$, and the symmetric bilinear form is given by $\langle v_1 \otimes v_2, v_1' \otimes v_2'\rangle = \langle v_1, v_1'\rangle \langle v_2, v_2' \rangle$, for any choice of Galois-equivariant symplectic form on $V_p^{\mathrm{std}}(f)$.

 We can, and do, choose the isomorphism $\psi$ of \eqref{eq:defpsi} to be compatible with the bilinear forms (up to scalars).

 \begin{proposition} For $i = 1, 2$, we have:
  \begin{enumerate}[(1)]
   \item The partial Frobenius $\varphi_i$ on $D_p(f)$ satisfies $\lambda(\varphi_i x, \varphi_i y) = p^{(w + 1)} \lambda(x, y)$.
   \item The space $\Fil_i^+ D_p$ is a maximal isotropic subspace of $D_p$.
   \item The eigenspaces $D_p^{(\varphi_i = \alpha_i)}$ and $D_p^{(\varphi_i = \beta_i)}$ are maximal isotropic.
  \end{enumerate}
 \end{proposition}

 \begin{proof}
  For part (1), we note that the isomorphism $\varphi_i^*(\cL_{\mu, \dR}) \cong \cL_{\mu, \dR}$, giving the action of $\varphi_i$ on the cohomology, multiplies the duality pairing on the fibres of $\cL_{\mu, \dR}$ by $p^w$ (since the pairing comes from the $w$-th tensor power of the Poincar\'e duality pairing on the top-degree cohomology of a Hilbert--Blumenthal abelian variety $A$, and the canonical isogeny $A \to A / (A[\pp] \cap \ker \varphi_A)$ has degree $p$, and thus acts as $p$ on the top-degree cohomology). Since $\varphi_i$ also has degree $p$ as a morphism from $Y_0$ to itself, it acts on the top-degree rigid cohomology as multiplication by $p^{w + 1}$.

   For (2), it follows easily from the shape of the cup-product on the BGG complex that $\Fil_i^+ D_p$ is an isotropic subspace of $D_p$, and since it is 2-dimensional it is maximal isotropic. For part (3), we note that assumption (c) implies $\alpha_i \ne \beta_i$; so we see from part (1) that $\lambda$ must vanish on the $\alpha_i$ and $\beta_i$ eigenspaces and identify each with the dual of the other. So they must each be 2-dimensional and maximal isotropic.
 \end{proof}

 \begin{corollary}
  If $k_1 \ne k_2$, the isomorphism $\psi_p$ respects the partial filtrations. If $k_1 = k_2$, then $\psi_p$ either respects the partial filtrations, or interchanges them (so the image of $\Fil_1^+ D_p$ is $D_{\pp_1} \otimes (\Fil^+  D_{\pp_2})$ and vice versa).
 \end{corollary}

 \begin{proof}
  Since $\psi_p$ must respect the Hodge filtration, we know that $\Fil_1^+ D_p + \Fil_2^+ D_p$ and $\Fil_1^+ D_p \cap \Fil_2^+ D_p$ map to their analogues in the tensor product. So $\psi_p$ descends to a bijection
  \[ \frac{\Fil_1^+ D_p + \Fil_2^+ D_p}{\Fil_1^+ \cap \Fil_2^+} \to \frac{(\Fil^+ D_{\pp_1} \otimes D_{\pp_2}) + (D_{\pp_1} \otimes \Fil^+ D_{\pp_2})}{\Fil^+ \otimes \Fil^+}. \]

  The orthogonal form $\lambda$ descends to a nondegenerate orthogonal form on $\frac{\Fil_1^+ D_p + \Fil_2^+ D_p}{\Fil_1^+ \cap \Fil_2^+}$. Since a nondegenerate orthogonal form cannot have more than two isotropic lines, it follows that $\Fil_1^+$ and $\Fil_2^+$ are the only two isotropic subspaces intermediate between $\Fil_1^+ + \Fil_2^+$ and $\Fil_1^+ \cap \Fil_2^+$. Since $\psi$ is compatible with the orthogonal forms, it follows that $\psi_p$ must map these spaces to $(\Fil^+ D_{\pp_1} \otimes D_{\pp_2})$ and $(D_{\pp_1} \otimes \Fil^+ D_{\pp_2})$ \emph{in some order}. If $k_1 \ne k_2$, then only one of these lines is contained in the middle Hodge filtration step, so we conclude that $\psi_p$ sends $\Fil_1^+ D_p$ to $\Fil^+ D_{\pp_1} \otimes D_{\pp_2}$, and similarly for $\Fil_2^+ D_p$.
 \end{proof}

 We now consider the Frobenius action, using hypothesis (c). Then the $(\varphi_1 = \alpha_1, \varphi_2 = \alpha_2)$ simultaneous eigenspace of $D_p$ coincides with the  $\varphi = \alpha_1 \alpha_2$ eigenspace. (Containment is clear, and hypothesis (c) implies that $D_p^{\varphi = \alpha_1 \alpha_2}$ has zero intersection with any of the other $(\varphi_1, \varphi_2)$-simultaneous eigenspaces; since the direct sum of these simultaneous eigenspaces is all of $D_p$, we must have equality.) Similarly, the $(\varphi \otimes 1 = \alpha_1, 1 \otimes \varphi = \alpha_2)$ simultaneous eigenspace of $D_{\pp_1} \otimes D_{\pp_2}$ coincides with the $\varphi_1 \otimes \varphi_2 = \alpha_1 \alpha_2$ eigenspace. Since $\psi_p$ commutes with $\varphi$, it must therefore send the $(\varphi_1 = \alpha_1, \varphi_2 = \alpha_2)$ simultaneous eigenspace to the $(\varphi \otimes 1 = \alpha_1, 1 \otimes \varphi = \alpha_2)$ simultaneous eigenspace. Repeating the argument for the other three pairs of roots, we conclude that $\psi_p$ commutes with the partial Frobenii. In particular, if $k_1 \ne k_2$ the proof of \cref{thm:nonBC} is complete.

 In the more delicate $k_1 = k_2$ case, we use \cref{thm:main}. Without loss of generality we suppose $\alpha_1$ has strictly small slope. Then the theorem tells us that $\Fil_2^+ D_p \cap D_p^{(\varphi_1 = \alpha_1)}$ is one-dimensional. However, since $\alpha_1$ has strictly small slope, its valuation is \emph{a fortiori} smaller than $k_1 + t_1 + 1$; so the weak admissibility of $D_{\pp_1}$ implies that $\Fil^+ D_{\pp_1} \cap D_{\pp_1}^{\varphi = \alpha_1}$ must be zero. Thus $(\Fil^+ D_{\pp_1}) \otimes D_{\pp_2}$ has zero intersection with the $\varphi \otimes 1 = \alpha_1$ eigenspace. As $\psi_p$ is compatible with the partial Frobenii, it cannot send $\Fil_2^+ D_p$ to $(\Fil^+ D_{\pp_1}) \otimes D_{\pp_2}$, since the former has nontrivial intersection with the $\varphi_1 = \alpha_1$ eigenspace, while the latter does not. So it must send $\Fil_2^+ D_p$ to $D_{\pp_1} \otimes (\Fil^+ D_{\pp_2})$, and  This completes the proof of \cref{thm:nonBC}.

\subsection{Repeated eigenvalues}

 We now consider, briefly, the contrary situation where (a), (b) of \cref{thm:nonBC} hold, but rather than (c), we suppose that the set of pairwise products has size exactly 3; so without loss of generality we may suppose that $\alpha_1 \beta_2 = \beta_1 \alpha_2$ but there are no other repetitions. Then the same proof as above shows that the $(\alpha_1, \alpha_2)$ and $(\beta_1, \beta_2)$ eigenspaces are sent to their analogues in $D_{\pp_1} \otimes D_{\pp_2}$. Meanwhile, the 2-dimensional $(\varphi = \alpha_1\beta_2 = \beta_1 \alpha_2)$-eigenspace is the sum of the $(\alpha_1, \beta_2)$ and $(\beta_1, \alpha_2)$ simultaneous eigenspaces, which are isotropic lines. So these must be sent to the $(\alpha_1, \beta_2)$ and $(\beta_1, \alpha_2)$ eigenspaces in the tensor product \emph{in some order}. That is, we have two cases:

 \begin{itemize}
  \item Case A: $\psi_p$ commutes with the partial Frobenii.
  \item Case B: $\psi_p$ intertwines $\varphi_1$ and $\varphi_2$ with  $\xi \cdot (1 \otimes \varphi)$ and $\xi^{-1} \cdot (\varphi \otimes 1)$, where $\xi = \tfrac{\alpha_1}{\alpha_2}$.
 \end{itemize}

 (Note that in case B we must have $\xi = \pm 1$, since $\psi_p$ is compatible with the bilinear forms, and $\varphi_1$ scales the bilinear form on $D_p$ by $p^{w + 1}$ while $\xi \cdot (1 \otimes \varphi)$ scales the bilinear form on $D_{\pp_1} \otimes D_{\pp_2}$ by $\xi^2 \cdot p^{w + 1}$.)

 In case A, we can argue exactly as before to show that $\psi_p$ commutes with the plectic structures. In case B, the same argument shows that we must have $k_1 = k_2$, and the isomorphism $\psi_p$ is an ``anti-plectic isomorphism'' (up to a twist), interchanging the roles of the two embeddings.

 \begin{remark}
  As noted above, the cases when $\alpha_1 = \beta_1$ or $\alpha_2 = \beta_2$ are conjectured never to occur. The only other case we have not considered is when $\alpha_1 / \beta_1 = \alpha_2 / \beta_2 = -1$; this definitely can arise, e.g.~if $f$ has complex multiplication by a totally-imaginary quadratic extension of $F$ in which both $\pp_i$ are inert, but it does not seem to be possible to treat it using the above methods.
 \end{remark}

 The above discussion applies, in particular, in the setting of \cref{thm:BC}: the assumptions (a') and (b') imply (a), (b) of \cref{thm:nonBC}, while (c') implies that the set
 \[ \{\alpha_1 \alpha_2,\, \dots,\, \beta_1 \beta_2\} = \{ \alpha_0^2,\, \alpha_0 \beta_0,\, \beta_0^2\}\]
 has three distinct elements. So $\psi_p$ must be either a plectic isomorphism, or a plectic anti-isomorphism. However, in this case $\psi$ is not itself uniquely determined (even up to scalars): since $f^{\sigma} = f$, the Galois representation $V_p(f)$ has an additional order 2 involution compatible with the orthogonal forms, which corresponds to swapping the factors of the tensor product. Replacing $\psi$ with its composite with this involution if necessary, we conclude that there is some isomorphism $\psi$ fulfilling the conditions of \cref{main}, thus proving \cref{thm:BC}.

 \begin{remark}
  If $f$ has $k_1 = k_2$ and $\alpha_1 \beta_2 = \beta_1 \alpha_2$, but $f$ is \emph{not} globally a twist of a base-change form (and is non-CM), then we are stuck. In this case, an isomorphism between $D_p$ and $D_{\pp_1} \otimes D_{\pp_2}$ which is compatible with the $\varphi$-module structure and the orthogonal forms must be either plectic, or anti-plectic (and both cases can occur). However, since $V_p(f)$ is irreducible in this case, only one of the two possibilities will be compatible with the global Galois action, and we cannot rule out the bizarre possibility that the global isomorphisms are the locally anti-plectic ones!
 \end{remark}

\subsection*{Acknowledgements} This paper is dedicated to the memory of our late colleague and friend Jan Nekov\'a\v{r}, whose work has been an inspiration to both of us over more than two decades -- from his undergraduate lectures in Cambridge which originally inspired Sarah to study number theory, to the email exchange with him during September of 2022 which was the starting-point for this paper.

It is also a pleasure to thank George Boxer, Vincent Pilloni, and Tony Scholl for informative conversations in connection with this paper; and the anonymous referees for their valuable comments.

\providecommand{\bysame}{\leavevmode\hbox to3em{\hrulefill}\thinspace}
\renewcommand{\MR}[1]{%
 MR \href{http://www.ams.org/mathscinet-getitem?mr=#1}{#1}.
}
\providecommand{\href}[2]{#2}
\newcommand{\articlehref}[2]{\href{#1}{#2}}

\end{document}